\documentclass[reqno]{amsart}
\usepackage[sort]{cite}

\usepackage{amsthm,bm,amsfonts,mathrsfs,graphicx}
\usepackage{hyperref}
\hypersetup{
    colorlinks=true,
    linkcolor=blue,
    filecolor=magenta,
    urlcolor=cyan
    }
\usepackage{appendix}
\usepackage{aliascnt}
\usepackage{mathtools}
\usepackage{lmodern}
\usepackage{xfrac}
\usepackage{cleveref}

\DeclareMathOperator{\Laplacian}{\mathop{}\!\mathbin\bigtriangleup}
\DeclareMathOperator{\R}{\mathbb{R}}

\DeclareMathOperator{\Hcal}{\mathcal{H}}
\DeclareMathOperator{\Ncal}{\mathcal{N}}
\DeclareMathOperator{\Kcal}{\mathcal{K}}
\DeclareMathOperator{\Lcal}{\mathcal{L}}
\DeclareMathOperator{\Lfrak}{\mathfrak{L}}
\DeclareMathOperator{\spn}{span}
\DeclarePairedDelimiterX{\Set}[2]\{\}{%
    \, #1 \;\delimsize\vert\; #2 \,
}

\newcommand\fct[2]{{#1}\left(#2\right)}
\newcommand\actn[2]{\left({#1},#2\right)}
\newcommand\ractn[2]{\langle{#1},#2\rangle}
\newcommand\abs[1]{\left|#1\right|}
\newcommand\Area[1]{{\bm\sigma}\left(#1\right)}
\newcommand{\Newton}[1]{\left|#1 \right|^4}
\newcommand{\sts}[1]{\left(#1\right)}
\newcommand{\qtq}[1]{\quad\text{#1}\quad}
\def\odif#1{\mathrm{d}#1}
\theoremstyle{plain}

\newtheorem*{prop*}{Proposition}
\newtheorem{theo}{Theorem}[section]
\crefname{theo}{theorem}{theorems}
\Crefname{theo}{Theorem}{Theorems}

\newaliascnt{lemma}{theo}
\newtheorem{lemm}[lemma]{Lemma}
\aliascntresetthe{lemma}
\crefname{lemm}{lemma}{lemma}
\Crefname{lemm}{Lemma}{Lemmas}

\newaliascnt{corollary}{theo}
\newtheorem{coro}[corollary]{Corollary}
\aliascntresetthe{corollary}
\crefname{coro}{corollary}{corollary}
\Crefname{coro}{Corollary}{Corollaries}

\newaliascnt{proposition}{theo}
\newtheorem{prop}[proposition]{Proposition}
\aliascntresetthe{proposition}
\crefname{prop}{proposition}{proposition}
\Crefname{prop}{Proposition}{Propositions}

\theoremstyle{definition}
\newtheorem{rema}[theo]{Remark}
\newtheorem*{rema*}{Remark}

\newtheorem*{nota*}{Notation}
\newtheorem*{notas*}{Notations}

\numberwithin{equation}{section}
\numberwithin{theo}{section}

\newcommand\Sphere[1]{\mathbb{S}^{#1}}

\begin{document}

\title[Nondegeneracy of positive solutions for NLH]{Nondegeneracy of the positive solutions for critical nonlinear Hartree equation in $\R^6$}

\author[]{Xuemei Li}
\address{\hskip-1.15em Xuemei Li
	\hfill\newline School of Mathematical Sciences,
	\hfill\newline Beijing Normal University,
	\hfill\newline Beijing, 100875, People's Republic of China.}
\email{xuemei\_li@mail.bnu.edu.cn}

\author[]{Xingdong Tang}
\address{\hskip-1.15em Xingdong Tang
	\hfill\newline School of Mathematics and Statistics, \hfill\newline Nanjing Univeristy of Information Science and Technology,
	\hfill\newline Nanjing, 210044,  People's Republic of China.}
\email{txd@nuist.edu.cn}

\author[]{Guixiang Xu}
\address{\hskip-1.15em Guixiang Xu
	\hfill\newline Laboratory of Mathematics and Complex Systems,
	\hfill\newline Ministry of Education,
	\hfill\newline School of Mathematical Sciences,
	\hfill\newline Beijing Normal University,
	\hfill\newline Beijing, 100875, People's Republic of China.}
\email{guixiang@bnu.edu.cn}



\begin{abstract}
We prove that any positive solution for the critical nonlinear Hartree equation
$$-\Laplacian\fct{u}{x}
-\int_{\R^6}
\frac{\abs{\fct{u}{y}}^2
}{
    \abs{x-y}^4
}\odif{y}
\,\fct{u}{x}=0,\qtq{} x\in\R^6.$$
is nondegenerate.
Firstly, in terms of spherical harmonics, we show that
the corresponding linear operator can be decomposed into a series of one dimensional linear operators.
Secondly, by making use of the Perron-Frobenius property, we show that the kernel
of each one dimensional linear operator is finite. Finally, we show that the kernel of
the corresponding linear operator is the direct sum of the kernel
of all one dimensional linear operators.
\end{abstract}
\maketitle

\section{Introduction}
The purpose of this paper is to derive the nondegeneracy property of the postive solutions
to the following $\dot{H}^1$-critical nonlinear Hartree (NLH) equation,
\begin{align}
    \label{NLH}
\begin{cases}
&
-\Laplacian\fct{u}{x}
-\int_{\R^6}
\frac{\abs{\fct{u}{y}}^2
}{
    \abs{x-y}^4
}\odif{y}
\,\fct{u}{x}=0,
    \quad x\in\R^6,
\\
&
\lim\limits_{{\abs{x}\to +\infty}}\abs{\fct{u}{x}}=0.
\end{cases}
\end{align}
It is well-known that (see, for example \cite{miao.wu.ea2015})
any postive solution of the equation \eqref{NLH} belongs precisely to
\begin{equation}
\label{omega family}
    \Set{ \fct{\omega_{\lambda,z}}{x}=\lambda^2\fct{\omega}{\lambda x+z}}{ \lambda>0\text{~and~}z\in \R^6 },
\end{equation}
where $\omega$ is the radial, postive ground state, and has the explicit form
\begin{equation}
\label{omega}
    \fct{\omega}{x}=\frac{12}{\pi^{\sfrac{3}{2}}}\frac{1}{\sts{1+\abs{x}^2}^2}.
\end{equation}

On the one hand, since the function $\fct{\omega_{\lambda,z}}{x}$,
defined by \eqref{omega family},
satisfies
\begin{align}
    \label{NLH scaling and translation}
        -\Laplacian\fct{\omega_{\lambda,z}}{x}
        -\int_{\R^6}
        \frac{\abs{\fct{\omega_{\lambda,z}}{y}}^2
        }{
            \abs{x-y}^4
        }\odif{y}\,\fct{\omega_{\lambda,z}}{x}=0,
        \quad x\in\R^6,~ \lambda>0, ~\text{and}~z\in \R^6,
\end{align}
by differentiating \eqref{NLH scaling and translation} with respect to
the parameters $\lambda$ and $z$ at $\lambda=1$ and $z=0$ formlly,
we obtain that,
\begin{equation*}
    \fct{L\Lambda\omega}{x}=0,\qtq{and}
    \fct{L\frac{\partial\omega}{\partial x_j}}{x}=0,
    \quad 1\leq j\leq 6,
\end{equation*}
where
\begin{equation}
\label{lambda omega}
    \fct{\Lambda\omega}{x}:=2\fct{\omega}{x}+2x\cdot\fct{\nabla\omega}{x},
\end{equation}
and the operator $L$ is defined by
\begin{align}
\label{linop}
    \fct{L\varphi}{x}:=-\fct{\Laplacian\varphi}{x}
    -\fct{\Phi}{\omega^2}\sts{x}\fct{\varphi}{x}
    -2\fct{\Phi}{\omega\varphi}\sts{x}\fct{\omega}{x},
\end{align}
with
\begin{equation}
\label{newton potential}
    \fct{\Phi f}{x}:= \int_{\R^6}\frac{\fct{f}{y}}{\abs{x-y}^4}\odif{y}.
\end{equation}
Moreover, by the linearity of the operator $L$, if $\varphi$ belongs to
the set
\begin{equation}
\label{kernel of L}
\Ncal=
\spn
    \left\{
    {
        \Lambda\omega,
    ~\frac{\partial\omega}{\partial x_1},
    ~\frac{\partial\omega}{\partial x_2},
    ~\frac{\partial\omega}{\partial x_3},
    ~\frac{\partial\omega}{\partial x_4},
    ~\frac{\partial\omega}{\partial x_5},
    ~\frac{\partial\omega}{\partial x_6}
    }
    \right\},
\end{equation}
we get, $\fct{L\varphi}{x}=0$.

On the other hand, an important question arising in the analysis of solutions for
\eqref{NLH} is to study the kernel of
the linearized operator $L$ close to $\omega$.
More precisely, one can address the following question:

\textit{
    Is there any other function $\varphi$ vanishing at infinity satisfies $\fct{L\varphi}{x}=0$,
    except that belongs to the set $\Ncal$ defined by \eqref{kernel of L}?
}

Our main result in this paper is devoted to a negative answer to the question aboved.
More precisely, we will prove the following theorem which states that
the solution $\omega$ (see \eqref{omega}) of \eqref{NLH}
is nondegenerate.
\begin{theo}
    \label{main theorem}
    The solution $\omega$, defined by \eqref{omega}, for the problem \eqref{NLH}
        is nondegenerate.
    More precisely, let $L$ be defined by \eqref{linop},
    if $f\in L^2\sts{\R^6}$ satisfies $\fct{{L f}}{x}=0$,
    then
\begin{equation*}
    f\in\spn
\left\{
        \Lambda\omega,
        ~\frac{\partial\omega}{\partial{x_1}},
         ~\frac{\partial\omega}{\partial{x_2}},
        ~\frac{\partial\omega}{\partial{x_3}},
        ~\frac{\partial\omega}{\partial{x_4}},
        ~\frac{\partial\omega}{\partial{x_5}},
         ~\frac{\partial\omega}{\partial{x_6}}
\right\}.
\end{equation*}
\end{theo}

The nondegeneracy of the ground state for the nonlinear elliptic equations
plays a key role in the analysis of long time dynamics of the solution to the corresponding
evaluation equations. For example, in the context of $\dot{H}^1$-critical nonlinear Schr{\"o}dinger (NLS) equation,
\begin{equation}
\label{NLS}
    i\fct{\frac{\partial u}{\partial t}}{t,x}+\fct{\Laplacian u}{t,x}+\fct{\abs{u}^{\frac{4}{N-2}}u}{t,x}=0,
\end{equation}
and the $\dot{H}^1$-critical nonlinear wave (NLW) equation,
\begin{equation}
\label{NLW}
    \fct{\frac{\partial^2 u}{\partial t^2}}{t,x}+\fct{\Laplacian u}{t,x}+\fct{\abs{u}^{\frac{4}{N-2}}u}{t,x}=0,
\end{equation}
the nondegeneracy of the ground state solution $W$ to the corresponding elliptic equation
\begin{equation}
\label{Elliptic equation:critical}
    -\fct{\Laplacian u}{t,x}-\fct{\abs{u}^{\frac{4}{N-2}}u}{t,x}=0,
\end{equation}
is crucial in the construction of blow-up solutions to the equations \eqref{NLS} and \eqref{NLW}
(see, for instance,
\cite{donninger.huang.ea2014,jendrej2019,jendrej.martel2020,krieger.nakanishi.ea2013,krieger.nakanishi.ea2015,martel.merle2016}).
With the help of \Cref{main theorem}, we are able to construct blow-up solutions
to the $\dot{H}^1$-critical nonlinear Schr{\"o}dinger equation with Hartree terms,
\begin{equation*}
    i\fct{\frac{\partial u}{\partial t}}{t,x}
    +\fct{\Laplacian u}{t,x}
    +\int_{\R^6}
    \frac{\abs{\fct{u}{y}}^2
    }{
        \abs{x-y}^4
    }\odif{y}
    \,\fct{u}{x}=0,
        \quad x\in\R^6,
\end{equation*}
which will be considered in our subsequent work.

The nondegeneracy of the ground state solution to the semilinear elliptic equation
\eqref{Elliptic equation:critical}
is also necessary in the construction of multi-bump solutions of the equation \eqref{Elliptic equation:critical},
see, for example
\cite{davila.delpino.ea2013,delpino.musso.ea2011,delpino.musso.ea2013,medina.musso.ea2019,medina.musso2021}.
We hope \Cref{main theorem} can be used to construct multi-bump solutions to \eqref{NLH}.

The paper is organized as follows. In \Cref{section:notations and lemmas},
we give some notations and review several lemmas which will be frequently used in the remainder of the paper.
In \Cref{section:proof of main results}, we mainly prove \Cref{main theorem}.



\section{Notation and useful lemmas}
\label{section:notations and lemmas}

\subsection*{Notation and conventions}
As usual, we use $\Sphere{5}$ to denote $5$-dimensional unit sphere in $6$-dimensional Euclidean space $\R^6$,
\begin{equation*}
    \Sphere{5}=\Set{x=\sts{x_1,x_2,\cdots,x_6}\in\R^6}{\abs{x}^2=\sum_{j=1}^6{x_j^2}=1}.
\end{equation*}
For any $x,y\in \R^6$ with $\abs{x}\neq\abs{y}$ , let us denote,
\begin{align}
\label{xvy:def}
    x\vee y
    :=
    \begin{cases}
        x & \qtq{if} \abs{x}>\abs{y},\\
        y & \qtq{if} \abs{x}<\abs{y},
    \end{cases}
    \qtq{and}
    x\wedge y
    :=
    \begin{cases}
        y & \qtq{if} \abs{x}>\abs{y},\\
        x & \qtq{if} \abs{x}<\abs{y}.
    \end{cases}
\end{align}
An elementary calculation implies that,
\begin{equation}
\label{xvy:abs}
    \abs{x\vee y}=\max\left\{\abs{x},\,\abs{y}\right\},
    \qtq{and}
    \abs{x\wedge y}=\min\left\{\abs{x},\,\abs{y}\right\}.
\end{equation}

We use $L^2\sts{\R^6}$ denote
the real Hilbert space of measurable functions $f$ on $\R^6$
with the inner product,
\begin{equation*}
    \actn{f}{g}:=\int_{\R^6}\fct{f}{x}\fct{g}{x}\odif{x}.
\end{equation*}
We shall also use $L^2\sts{\Sphere{5}}$ denote
the space of measurable functions $f$ on $\Sphere{5}$ for
which
$$
\int_{\Sphere{5}}\abs{\fct{f}{x}}^2\odif{\Area{x}}
$$
is finite, where $\odif{\bm \sigma}$ is the surface area measure.
With a slight abuse of notation of notation,
we write both $\fct{f}{x}$ and $\fct{f}{\abs{x}}$ for radial functions $f$ on $\R^6$.
Moreover, we will use $L^2\sts{\sts{0,+\infty},r^5}$ denote the Hilbert space
with measurable functions $f$ on $ \sts{0,+\infty} $ with the inner product,
\begin{equation*}
    \ractn{f}{g}:=\int_0^{+\infty}\fct{f}{r}\fct{g}{r}r^5\odif{r}.
\end{equation*}

Next, we recall some well-known results related to spherical harmonics.
We use $\Hcal_{k}$ to denote the space of spherical harmonics of degree $k$
(i.e. the restrictions to $\Sphere{5}$ of real, homogeneous harmonic
polynomials of degree $k$).
In fact, the dimension of $\Hcal_{k}$ is\cite{dai.xu2013,stein.weiss1975} is
\begin{equation*}
    \dim{\Hcal_{k}}=\alpha_k:=
    \begin{cases}
        1, & \qtq{if} k=0,\\
        6, & \qtq{if} k=1,\\
        \binom{k+5}{k}-\binom{k+3}{k-2} & \qtq{if} k\geq 2.
    \end{cases}
\end{equation*}
We use ${Y_{k,j}}\sts{1\leq j\leq \alpha_k}$
to denote an orthogonal basis for $\Hcal_{k}$, i.e.
\begin{equation}
\label{orthogonality of basis 1}
    \int_{\Sphere{5}}\fct{Y_{k,i}}{\xi}\fct{Y_{k,j}}{\xi}\odif{\Area{\xi}}=\pi^3
\begin{cases}
        1, &         \text{if } i=j,\\
        0, &         \text{if } i\neq j.
\end{cases}\footnote{ $\pi^3$ is the volume of $\Sphere{5}$. }
\end{equation}
Specially, $\Hcal_{0}=\spn\left\{1\right\}$, and
\begin{equation}
\label{spherical harmonics space of degree 1}
    \Hcal_{1}=\spn\Set{ \frac{x_j}{\abs{x}} }{1\leq j\leq 6},
\end{equation}
therefore, we may set $\fct{Y_{1,j}}{\xi}=\sqrt{6}\xi_j$ with $\xi_j=\frac{x_j}{\abs{x}}$.
Moreover, the space $\Hcal_{k}$ coincides with
the eigenspace of the eigenvalue $-k\sts{k+4}$
for the Laplace–Beltrami operator $\Laplacian_{\Sphere{5}}$ on $\Sphere{5}$, i.e.
\begin{equation}
\label{eigen equation}
    \Laplacian_{\Sphere{5}}Y = -k\sts{k+4} Y,\qtq{for any} Y\in \Hcal_{k}.
\end{equation}
An elementary calculation implies that,
\begin{equation}
\label{orthogonality of basis 2}
    \int_{\Sphere{5}}\fct{Y_{k,i}}{\xi}\fct{Y_{l,j}}{\xi}\odif{\xi}
    =0,\quad \text{for all } 1\leq k\neq l<\infty,~~1\leq i\leq \alpha_k,~~1\leq j\leq \alpha_l.
\end{equation}
Now, for any $f\in L^2\sts{\R^6}$, let us denote
\begin{equation}
    \label{spherical harmonics coefficient}
    \fct{f_{k,j}}{r}
    :=\frac{1}{\pi^3}\int_{\Sphere{5}}
    \fct{f}{r\xi}\fct{Y_{k,j}}{\xi} \odif{\Area{\xi}},
\end{equation}
then we have the following direct sum decomposition (e.g. see \cite{stein.weiss1975}),
\begin{align}
\label{spherical harmonics expand}
    \fct{f}{x}
    =
    \sum_{k=0}^{\infty}
    \sum_{j=1}^{\alpha_k}\fct{f_{k,j}}{\abs{x}}\fct{Y_{k,j}}{\frac{x}{\abs{x}}}.
\end{align}
Moreover, the following identity holds,
\begin{align}
\label{parseval identity}
    &\int_{\R^6}\abs{\fct{f}{x}}^2\odif{x}
    =
    \pi^3\sum_{k=0}^{\infty}\sum_{j=1}^{\alpha_k}
    \int_{0}^{+\infty}
    \abs{
        \fct{\varphi_{k,j}}{r}
    }^2r^5\odif{r}.
\end{align}

In order to deal with the Newtonian potential appeared in \eqref{NLH}, we need the following lemma,
 which related to Gegenbauer functions.
\begin{lemm}[{\cite[Lemma 1.2.3,~Theorem1.2.6]{dai.xu2013}}]
For any  $r\in\sts{-1,\,1}$, $t\in\left[-1,\,1\right]$, we have,
\begin{equation}
\label{newtonian to gegenbauer}
    \frac{1}{\sts{1-2rt+r^2}^2}
     =\sum_{k=0}^{\infty}\fct{C^{(2)}_k}{t}r^k,
\end{equation}
where $\fct{C^{(2)}_k}{t}$ is the Gegenbauer
polynomial of degree $k$ associated with $2$ (see \cite[p.418]{dai.xu2013}).
Moreover, the following addition formula holds,
\begin{equation}
\label{gegenbauer to spherical}
    \fct{P^2_k}{\xi\cdot\eta}
    =
    \frac{2}{k+2}\sum_{j=1}^{\alpha_k}
    \fct{Y_{k,j}}{\xi}\fct{Y_{k,j}}{\eta},
    \qtq{for any} \xi,~\eta\in\Sphere{5}.
\end{equation}
\end{lemm}

Now, we are able to give an explicit expression of the Newtonian potential in $\R^6$,
by making use of spherical harmonics.
\begin{lemm}
\label{lem:newton potential}
    For any $x$, $y\in \R^6\setminus\{0\}$ with $\abs{x}\neq\abs{y}$, we have,
\begin{align}
\label{eq:newton potential}
    \frac{1}{ \abs{x-y}^4 }
    =
    \sum_{k=0}^{\infty}\frac{k}{k+2}\frac{ \abs{x\wedge y}^k }{ \abs{x\vee y}^{k+4} }
    \sum_{m=1}^{\alpha_k}\fct{Y_{k,j}}{\frac{x}{\abs{x}}}\fct{Y_{k,j}}{\frac{y}{\abs{y}}},
\end{align}
\end{lemm}
\begin{proof}
    First, for any $x$, $y\in \R^6\setminus\{0\}$ satisfying $\abs{x}\neq\abs{y}$, we have,
\begin{align}
\label{np01}
    \abs{x-y}^2=\abs{x}^2-2x\cdot y+\abs{y}^2.
\end{align}
By making use of \eqref{xvy:def} and \eqref{xvy:abs}, we have, from \eqref{np01},
\begin{align}
\label{np02}
    \abs{x-y}^2
    =
    \abs{x\vee y}^2
    -2\abs{x\vee y}\abs{ x\wedge y}\frac{x}{\abs{x}}\cdot \frac{y}{\abs{y}}
    +\abs{x\wedge y}^2.
\end{align}
Therefore, by \eqref{np02}, we obtain that,
\begin{align}
\label{np03}
    \frac{1}{\abs{x-y}^4}
    =
    \frac{1}{ \abs{x\vee y}^4 }
    \frac{1}{\sts{ 1-2\frac{\abs{ x\wedge y}}{\abs{x\vee y}}
    \frac{x}{\abs{x}}\cdot \frac{y}{\abs{y}}
    +\sts{ \frac{\abs{ x\wedge y}}{\abs{x\vee y}} }^2 }^2 }.
\end{align}
Now, by taking $r=\frac{\abs{ x\wedge y}}{\abs{x\vee y}}$ in \eqref{newtonian to gegenbauer},
then taking $\xi=\frac{x}{\abs{x}}$ and $\eta=\frac{y}{\abs{y}}$ in \eqref{gegenbauer to spherical},
The formula \eqref{eq:newton potential} follows directly from \eqref{np03}. This ends the proof of
\Cref{lem:newton potential}.
\end{proof}
\begin{rema}
   The spherical harmonics expansion for the Newtonian potential in $\R^n(3\leq n\leq 4)$ and $\R^4$ is well-known,
(see, for example, \cite{chen2021,lenzmann2009}).
\end{rema}

The following proposition plays an important role in the nondegenerate analysis of positive solutions to \eqref{NLH}.
\begin{prop}
\label{prop:newtonian convolution}
    Let $\Phi$ be defined by \eqref{newton potential}.
    For any $f\in L^2\sts{\R^6}$, we have,
\begin{align}
\label{eq:newtonian convolution}
    \fct{\Phi f}{x}
    =
    \sum_{k=0}^{\infty}
    \sum_{j=1}^{\alpha_k}
    \int_{0}^{+\infty}\fct{\Kcal_k}{r,\abs{x}}\fct{f_{k,j}}{r}\odif{r}
    \fct{Y_{k,j}}{\frac{x}{\abs{x}}},
\end{align}
where
$\fct{f_{k,j}}{r}$ is defined by \eqref{spherical harmonics coefficient},
and
\begin{equation}
\label{integral kernel}
    \fct{\Kcal_k}{r,\abs{x}}
    =
    \frac{2\pi^3}{k+2}
    \begin{cases}
        \frac{r^{k+5}}{\abs{x}^{k+4}},&\qtq{if} r<\abs{x},\\
        \frac{\abs{x}^{k}}{r^{k-1}},&\qtq{if} r>\abs{x}.
    \end{cases}
\end{equation}
\end{prop}
\begin{proof}
    On the one hand, in view of \Cref{lem:newton potential},
\begin{align}
\label{nc:001}
     \frac{1}{ \abs{x-y}^4 }
     =
     \sum_{k=0}^{\infty}\frac{k}{k+2}\frac{ \abs{x\wedge y}^k }{ \abs{x\vee y}^{k+4} }
    \sum_{m=1}^{\alpha_k}\fct{Y_{k,j}}{\frac{x}{\abs{x}}}\fct{Y_{k,j}}{\frac{y}{\abs{y}}},
\end{align}
where $x$, $y\in \R^6\setminus\{0\}$ satisfy $\abs{x}\neq\abs{y}$.

On the other hand, by \eqref{spherical harmonics expand}, we have
\begin{equation}
    \label{nc:002}
    \fct{f_{k,j}}{r}
    =\frac{1}{\pi^3}\int_{\Sphere{5}}
    \fct{f}{r\xi}\fct{Y_{k,j}}{\xi} \odif{\Area{\xi}},
\end{equation}
where
$\fct{f_{k,j}}{r}$ is defined by \eqref{spherical harmonics coefficient}.

Combining with \eqref{nc:001} and \eqref{nc:002}, we have,

\begin{align}
\nonumber
    &
    \int_{\R^6}\frac{\fct{f}{y}}{\Newton{x-y}}\odif{y}
    \\
\label{nc:003}
    =&
    \sum_{k=0}^{\infty}\sum_{l=0}^{\infty}\sum_{j=1}^{\alpha_k}\sum_{m=1}^{\alpha_l}
    \int_{\R^6}
    \frac{k}{k+2}\frac{ \abs{x\wedge y}^k }{ \abs{x\vee y}^{k+4} }
    \fct{Y_{k,j}}{\frac{x}{\abs{x}}}\fct{Y_{k,j}}{\frac{y}{\abs{y}}}
    \fct{f_{l,m}}{\abs{y}}\fct{Y_{l,m}}{\frac{y}{\abs{y}}}
    \odif{y}.
\end{align}
By the definition of $x\wedge y$ and $x\vee y$ (see \eqref{xvy:def}),
and \eqref{xvy:abs}, from \eqref{nc:003}, we find,
\begin{align}
    &
\nonumber
    \int_{\R^6}\frac{\fct{f}{y}}{\Newton{x-y}}\odif{y}
    \\
\nonumber
    =&
    \sum_{k=0}^{\infty}\sum_{l=0}^{\infty}\sum_{m=1}^{\alpha_k}\sum_{j=1}^{\alpha_k}
    \frac{k}{k+2}
    \int_{\abs{y}<\abs{x}}
    \frac{ \abs{y}^k }{ \abs{x}^{k+4} }
    \fct{Y_{k,j}}{\frac{x}{\abs{x}}}\fct{Y_{k,j}}{\frac{y}{\abs{y}}}
    \fct{f_{l,m}}{\abs{y}}\fct{Y_{l,m}}{\frac{y}{\abs{y}}}
    \odif{y}
    \\
\label{nc:004}
    &+
    \sum_{k=0}^{\infty}\sum_{l=0}^{\infty}\sum_{m=1}^{\alpha_k}\sum_{j=1}^{\alpha_k}
    \frac{k}{k+2}\int_{\abs{y}>\abs{x}}
    \frac{ \abs{x }^k }{ \abs{  y}^{k+4} }
    \fct{Y_{k,j}}{\frac{x}{\abs{x}}}\fct{Y_{k,j}}{\frac{y}{\abs{y}}}
    \fct{f_{l,m}}{\abs{y}}\fct{Y_{l,m}}{\frac{y}{\abs{y}}}
    \odif{y}.
\end{align}
With the change of variables $r=\abs{x}$ and $\xi=\frac{x}{\abs{x}}$
and using \eqref{orthogonality of basis 1}, \eqref{orthogonality of basis 2}, \eqref{nc:004}
becomes
\begin{align*}
    &
    \int_{\R^6}\frac{\fct{f}{y}}{\Newton{x-y}}\odif{y}
    \\
    =&
    \sum_{k=0}^{\infty}\sum_{j=1}^{\alpha_k}
    \frac{k}{k+2}\fct{Y_{k,j}}{\frac{x}{\abs{x}}}
    \left(
    \int_{0}^{\abs{x}}
    \pi^3
    \frac{r^{k+5}}{\abs{x}^{k+4}}\fct{f_{k,j}}{r}\odif{r}
    +
    \int_{\abs{x}}^{+\infty}
    \pi^3
    \frac{ \abs{x }^k }{ r^{k-1} }\fct{f_{k,j}}{r}\odif{r}
    \right),
\end{align*}
which means that \eqref{eq:newtonian convolution} holds
with $\fct{\Kcal_k}{r,\abs{x}}$ defined by \eqref{integral kernel}.
\end{proof}




As a consequence of \Cref{prop:newtonian convolution}, we obtain
Newton's theorem (see, for instance, \cite[Theorem 9.7]{lieb.loss2001}).
\begin{coro}
\label{coro:radial Newton's theorem}
    For any radial $f\in L^2\sts{\R^6}$,
$\fct{\Phi f}{x}$
is radial. Moreover, we have,
\begin{equation}
\label{eq:radial Newton's theorem}
    \fct{\Phi f}{x}
    =
    \int_{0}^{+\infty}\fct{\Kcal_0}{r,\abs{x}}\fct{f}{r}\odif{r},
\end{equation}
where $\fct{\Kcal_0}{r,\abs{x}}$ is defined by \eqref{integral kernel} with $k=0$.
Specially, the follows identity holds,
\begin{align}
\label{eq:newtonian potential for omega2}
    \fct{\Phi}{\omega^2}\sts{x}=2\pi^{\sfrac{3}{2}}\fct{\omega}{x},
\\
\label{eq:omega NLS}
-\Laplacian\fct{\omega}{x}-2\pi^{\sfrac{3}{2}}\fct{\omega}{x}=0.
\end{align}
\end{coro}
\begin{proof}
Since $f$ is radial, by \eqref{spherical harmonics coefficient}, we have
\begin{equation*}
    \fct{f_{0,1}}{r} = \fct{f}{r},
    \qtq{and}
    \fct{f_{k,j}}{r}\equiv 0,\qtq{for all} k\geq 1~\text{and}~1\leq j\leq \alpha_{k}.
\end{equation*}
Therefore, \eqref{eq:radial Newton's theorem} follows from \eqref{eq:newtonian convolution}.
The identity \eqref{eq:newtonian potential for omega2} follows from
elementary calculations and \eqref{eq:radial Newton's theorem}
by taking $\fct{f}{r}=\fct{\omega^2}{r}$.
Moreover, by inserting \eqref{eq:newtonian potential for omega2} into \eqref{NLH scaling and translation} with
$\lambda=1$ and $z=0$, we obtain that \eqref{eq:omega NLS} holds.
\end{proof}

The next theorem concerns the decomposition of the operator $L$ defined by \eqref{linop}
in terms of spherical harmonics.
\begin{theo}
\label{theo:decomposition L}
    Let $L$ be defined by \eqref{linop}. For any $ f\in L^2\sts{\R^6} $, we have,
\begin{align}
\label{eq:decomposition L}
    \fct{Lf}{x}
    =
    \sum_{k=0}^{\infty}\sum_{j=1}^{\alpha_k}
    \fct{\Lcal_{k}f_{k,j}}{\abs{x}}\fct{Y_{k,j}}{\frac{x}{\abs{x}}},
\end{align}
where
$f_{k,j}$ is defined by \eqref{spherical harmonics coefficient}, and
\begin{align}
\nonumber
    \fct{\Lcal_{k}f}{r}
    =
    &
    -\fct{f''}{r}
        -\frac{5}{r}\fct{f'}{r}
        +\frac{k\sts{k+4}}{r^2}\fct{f}{r}
    \\
    &
    -2\pi^{\sfrac{3}{2}} \fct{\omega}{r}\fct{f}{r}
    -2\fct{\omega}{r}\int_{0}^{+\infty}\fct{\Kcal_k}{t,r}\fct{\omega}{t}\fct{ f }{t}\odif{t},
\label{Lk operator}
\end{align}
with ${\Kcal_k}$ defined by \eqref{integral kernel}.
\end{theo}
\begin{proof}
First, by \eqref{spherical harmonics expand} and \eqref{spherical harmonics coefficient},
for any $ f\in L^2\sts{\R^6} $, we have,
\begin{align}
\label{dl:001}
    \fct{f}{x}
    =
    \sum_{k=0}^{\infty}\sum_{j=1}^{\alpha_k}
    \fct{f_{k,j}}{\abs{x}}\fct{Y_{k,j}}{\frac{x}{\abs{x}}},
\end{align}
where
\begin{equation}
    \label{dl:002}
    \fct{f_{k,j}}{r}
    =\frac{1}{\pi^3}\int_{\Sphere{5}}
    \fct{f}{r\xi}\fct{Y_{k,j}}{\xi} \odif{\Area{\xi}}.
\end{equation}
Using the fact that,
$$
\Laplacian \sts{\fct{f_{k,j}}{\abs{x}}\fct{Y_{k,j}}{\frac{x}{\abs{x}}}}
=
\fct{f_{k,j}''}{\abs{x}}+\frac{5}{\abs{x}}\fct{f_{k,j}'}{\abs{x}}
+\frac{1}{\abs{x}^2}\Laplacian_{\Sphere{5}}\fct{Y_{k,j}}{\frac{x}{\abs{x}}},
$$
and \eqref{eigen equation}, we obtain that,
\begin{align}
    \label{dl:006}
    -\Laplacian\fct{f}{x}
    =
    \sum_{k=0}^{\infty}\sum_{j=1}^{\alpha_k}
    \left(
        -\fct{f_{k,j}''}{\abs{x}}
        -\frac{5}{\abs{x}}\fct{f_{k,j}'}{\abs{x}}
        +\frac{k\sts{k+4}}{\abs{x}^2}\fct{f_{k,j}}{\abs{x}}
    \right)\fct{Y_{k,j}}{\frac{x}{\abs{x}}}.
\end{align}

Next, since $\omega$ is radial (see \eqref{omega}), by \Cref{coro:radial Newton's theorem},
we have,
\begin{align}
\label{dl:005}
    -\fct{\Phi}{\omega^2}\sts{x}\fct{f}{x}
    =
    -\sum_{k=0}^{\infty}\sum_{j=1}^{\alpha_k}
    2\pi^{\sfrac{3}{2}} \fct{\omega}{\abs{x}}
    \fct{f_{k,j}}{\abs{x}}\fct{Y_{k,j}}{\frac{x}{\abs{x}}}.
\end{align}

Now, by \eqref{prop:newtonian convolution}, we have,
\begin{align}
    -2\fct{\Phi}{\omega f}\sts{x}\fct{\omega}{x}
\label{dl:003}
    =
    -2
    \sum_{k=0}^{\infty}
    \sum_{j=1}^{\alpha_k}
    \int_{0}^{+\infty}\fct{\Kcal_k}{r,\abs{x}}\fct{\sts{ \omega f }_{k,j}}{r}\odif{r}
    \fct{Y_{k,j}}{\frac{x}{\abs{x}}}\fct{\omega}{\abs{x}}.
\end{align}
Since $\omega$ is radial, it follows from \eqref{spherical harmonics coefficient} that,
\begin{equation*}
    \fct{\sts{ \omega f }_{k,j}}{r} = \fct{\omega}{r}\fct{ { f }_{k,j}}{r},
\end{equation*}
which, together with \eqref{dl:003}, implies that,
\begin{align}
    -2\fct{\Phi}{\omega f}\sts{x}\fct{\omega}{x}
    \label{dl:004}
    =
    -2
    \sum_{k=0}^{\infty}
    \sum_{j=1}^{\alpha_k}
    \int_{0}^{+\infty}
    \fct{\Kcal_k}{r,\abs{x}}\fct{\omega}{r}\fct{ { f }_{k,j}}{r}\odif{r}
    \fct{Y_{k,j}}{\frac{x}{\abs{x}}}\fct{\omega}{\abs{x}}.
\end{align}

Combining \eqref{dl:006}, \eqref{dl:005} and \eqref{dl:004}, we obtain that
\eqref{eq:decomposition L} holds with $\Lcal_{k}$ defined by \eqref{Lk operator}.
\end{proof}

\begin{lemm}
\label{lem:Lk monotone}
        Let $\Lcal_{k}$ be defined by \eqref{Lk operator}.
    For any $f\geq 0$ with $f\neq 0$, we have,
\begin{align}
    \label{eq:Lk monotone}
        \ractn{\Lcal_{k}f}{f}
        >
        \ractn{\Lcal_{k-1}f}{f}.
\end{align}
\end{lemm}
\begin{proof}
        By an elementary computation, we have,
\begin{align*}
        &\ractn{\Lcal_{k}f}{f}-\ractn{\Lcal_{k-1}f}{f}
    \\
    =&
    \int_{0}^{+\infty}\sts{k\sts{k+4}-\sts{k-1}\sts{k+3}}
        \fct{f^2}{r}r^3\odif{r}
    \\
    &+
    2\pi^3
    \int_{0}^{+\infty}\frac{\fct{\omega}{r}\fct{f}{r}}{r^{k-1}}
    \int_{0}^{r}
    \sts{
        r-t
    }\fct{\omega}{t}t^{k+4}\fct{f}{t}\odif{t}
    \odif{r}
    \\
    &
    +
    2\pi^3
    \int_{0}^{+\infty}{\fct{\omega}{r}\fct{f}{r}}{r^{k+4}}
    \int_{r}^{+\infty}
    \sts{
        r-t
    }\frac{\fct{\omega}{t}\fct{f}{t}}{t^{k-1}}\odif{t}
    \odif{r},
\end{align*}
    which, combining with $f\geq 0$ and $f\neq 0$ shows that, \eqref{eq:Lk monotone} holds.
\end{proof}

The following lemma shows that the operator $L$, defined by \eqref{linop}, is nonnegative
under suitable orthogonal condition.
\begin{lemm}[\cite{miao.wu.ea2015}]
\label{lem:nonnegative}
Let $L$ be defined by \eqref{linop} and $\omega$ be defined by \eqref{omega}.
If $f\in H^1\sts{\R^6}$ satisfies $\actn{\nabla f}{\nabla\omega}=0$,
then $\actn{Lf}{f}\geq 0$.
\end{lemm}
As a consequence of \Cref{lem:nonnegative},
using \eqref{eq:omega NLS},
we obtain the following results, which will be important in our analysis.
\begin{coro}
\label{lem:nonnegative1}
    Let $L$ be defined by \eqref{linop} and $\omega$ be defined by \eqref{omega}. If $f\in L^2\sts{\R^6}$
satisfies $Lf\in L^2\sts{\R^6}$ and $\actn{ f}{\omega^2}=0$, then we have $\actn{L f}{f}\geq 0$.
\end{coro}
\begin{coro}
\label{coro:nonnegative}
Let $L$ be defined by \eqref{linop}, $Y\in \Hcal_{k}$ for some $k\geq 1$,
and $g\in L^2\sts{\sts{0,+\infty}, r^5}$. If
$\fct{f}{x}=\fct{g}{\abs{x}}\fct{Y}{\frac{x}{\abs{x}}}$ satisfies $Lf\in L^2\sts{\R^6}$,
then we have $\actn{L f}{f}\geq 0$.
\end{coro}
\begin{proof}
First, since $\omega^2$ is radial, by changing the variables $r=\abs{x}$ and $\xi=\frac{x}{\abs{x}}$, we have,
\begin{equation*}
    \int_{\R^6}\fct{f}{x}\fct{\omega^2}{x}\odif{x}
    =
    \int_{0}^{+\infty}\fct{f}{r}\fct{\omega^2}{r}r^5\odif{r}\int_{\Sphere{5}}\fct{Y}{\xi}\odif{\Area{\xi}},
\end{equation*}
which, combining \eqref{orthogonality of basis 1} and \eqref{orthogonality of basis 2}, implies that
$\actn{f}{\omega^2}=0$. Therefore, by \Cref{lem:nonnegative1}, we have $\actn{L f}{f}\geq 0$.
\end{proof}

\section{Proof of main results}
\label{section:proof of main results}
In this section, we prepare the proof of \Cref{main theorem} by means of
several propositions.
The following proposition shows that
the kernel of the operator \eqref{linop} restricted to the radial functions
is $\spn\left\{\Lambda\omega\right\}$.
\begin{prop}
\label{prop:kernel of L0}
Let $\Lcal_{0}$ be defined by \eqref{Lk operator} with $k=0$.
If $\varphi\in L^2\sts{\sts{0,+\infty},r^5}$ with $\varphi\not\equiv 0$ satisfies $\Lcal_{0}\varphi=0$,
then there exists  $\eta\in\R$, such that
\begin{equation}
\label{kernel of L0}
    \fct{\varphi}{r}=\eta \fct{\Lambda\omega}{r},
\end{equation}
where $\Lambda\omega$ is defined by \eqref{lambda omega}.
\end{prop}

\begin{proof}
First of all, for any $\varphi\in L^2\sts{\sts{0,+\infty},r^5}$ satisfies $\Lcal_{0}\varphi=0$,
a bootstrap argument implies that $\varphi$ is smooth and $\fct{\varphi'}{0}=0$.

Now, we cliam that,
\begin{equation}
\label{kl0:002}
    \fct{\Lcal_{0}\Lambda\omega}{r}=0.
\end{equation}
Indeed, for all $\lambda>0$,
$\fct{\omega_{\lambda}}{x}:=\lambda^2\fct{\omega}{\lambda x}$
satisfy \eqref{NLH}, i.e.
\begin{align}
\label{kl0:001}
    -\Laplacian\fct{\omega_{\lambda}}{x}
    -\int_{\R^6}
    \frac{\abs{\fct{\omega_{\lambda}}{y}}^2
    }{
        \abs{x-y}^4
    }\odif{y}\,\fct{\omega_{\lambda}}{x}=0,
    \quad x\in\R^6,
\end{align}
by differentiation equation \eqref{kl0:001} with respect to $\lambda$, we obtain that,
\begin{align*}
    \fct{L\Lambda\omega}{x}=0,\quad x\in\R^6,
\end{align*}
where $L$ is defined by \eqref{linop}.
Moreover, since $\fct{\Lambda\omega}{x}$ is radial,
it follows from \eqref{spherical harmonics coefficient} that,
\begin{equation*}
    \fct{\sts{\Lambda\omega}_{0,1}}{r}=\fct{\Lambda\omega}{r},
    \qtq{and}
    \fct{\sts{\Lambda\omega}_{k,j}}{r}=0, \quad\text{for all~~}
    k\geq 1 \text{~~and~~} 1\leq j\leq \alpha_k,
\end{equation*}
which, together with \Cref{theo:decomposition L} implies that \eqref{kl0:002} holds.

Next, let us rewirte $\fct{\Lcal_{0}\varphi}{r}$ as
\begin{align}
\label{decomposition of L0}
    \fct{\Lcal_{0}\varphi}{r}
    =
    \fct{\Lfrak_{0}\varphi}{r}
    -2\pi^3\fct{\omega}{r}\int_{0}^{+\infty}\fct{\varphi}{t}\fct{\omega}{t}t\odif{t},
\end{align}
where
\begin{align}
\label{operator L0:part1}
    \fct{\Lfrak_{0}\varphi}{r}
    =
    &
    -\fct{\varphi''}{r}
        -\frac{5}{r}\fct{\varphi'}{r}
    -2\pi^{\sfrac{3}{2}} \fct{\omega}{r}\fct{\varphi}{r}
    -2\pi^3\int_0^r\sts{ \frac{t^4}{r^4}-1 }\fct{\varphi}{t}\fct{\omega}{t}t\odif{t}.
\end{align}
Therefore,
for any $\varphi\in L^2\sts{\sts{0,+\infty},r^5}$
with $\varphi\not\equiv 0$ satisfying $\Lcal_{0}\varphi=0$,
by \eqref{decomposition of L0} and \eqref{operator L0:part1}, we have,
\begin{align}
\label{kl0:003}
    \fct{\Lfrak_{0}\varphi}{r}
    =2\pi^3\fct{\omega}{r}\int_{0}^{+\infty}\fct{\varphi}{t}\fct{\omega}{t}t\odif{t},
\end{align}
and specially,
\begin{align}
\label{kl0:004}
    \fct{\Lfrak_{0}\Lambda\omega}{r}
    =2\pi^3\fct{\omega}{r}\int_{0}^{+\infty}\fct{\Lambda\omega}{t}\fct{\omega}{t}t\odif{t}.
\end{align}
By setting
\begin{align*}
    \fct{\phi}{r}
    =
    \fct{\varphi}{r}
    -
    \frac{
        \int_{0}^{+\infty}\fct{\varphi}{t}\fct{\omega}{t}t\odif{t}
    }{
        \int_{0}^{+\infty}\fct{\Lambda\omega}{t}\fct{\omega}{t}t\odif{t}
    }\fct{\Lambda\omega}{r},
\end{align*}
we have $\phi\in L^2\sts{\sts{0,+\infty},r^5}$ is smooth and $\fct{\phi'}{0}=0$,
Moreover, an elementary calculation implies that,
\begin{equation*}
    \fct{\Lfrak_{0}\phi}{r}=0.
\end{equation*}
Now, by \Cref{prop:kernel operator L0:part1}, we have
$$\fct{\phi}{r}=0 \qtq{for all} r\in\left[0,+\infty\right), $$
which implies that \eqref{kernel of L0} holds with
$\eta=\frac{
    \int_{0}^{+\infty}\fct{\varphi}{t}\fct{\omega}{t}t\odif{t}
}{
    \int_{0}^{+\infty}\fct{\Lambda\omega}{t}\fct{\omega}{t}t\odif{t}
}$. This ends the proof of \Cref{prop:kernel of L0}.


\end{proof}


\begin{prop}
\label{prop:Perron-Frobenius}
Let  $\Lcal_{k}$ be defeind by \eqref{Lk operator}.
For each $k\geq 1$,  the operator $\Lcal_{k}$
is bounded below and essentially self-adjoint
on $C_0^{\infty}\sts{0,+\infty}\subset L^2\sts{\sts{0,+\infty},r^5\odif{r}}$.
Moreover, for each $k\geq 1$, the operator $\Lcal_{k}$
enjoys the Perron-Frobenius property, i.e. if
\begin{equation*}
    \lambda_{k}^{0}=\inf\Set{\ractn{\Lcal_{k}f}{f}}{\int_{0}^{+\infty}\abs{\fct{f}{r}}^2r^5\odif{r}=1 }
\end{equation*}
is attained, then the lowest eigenvalue $\lambda_{k}^{0}$ of the operator $\Lcal_{k}$ is simple,
and the corresponding eigenfunction $\fct{\chi_{k}^{0}}{r}$ dose not change sign on $\sts{0,+\infty}$.
\end{prop}
\begin{proof}
    The proof of \Cref{prop:Perron-Frobenius} is almost identical to that of \cite[Lemma 7]{lenzmann2009},
therefore we omit the details.
\end{proof}
\begin{prop}
\label{prop:kernel of L1}
Let $\Lcal_{1}$ be defined by \eqref{Lk operator} with $k=1$.
If $\psi\in L^2\sts{\sts{0,+\infty},r^5}$ with $\psi\not\equiv 0$ satisfies $\Lcal_{1}\psi=0$,
then there exists  $\beta\in\R$, such that
\begin{equation}
\label{kernel of L1}
    \fct{\psi}{r}=\beta \fct{\omega'}{r}.
\end{equation}
Moreover,
\begin{equation}
\label{lowest eigenvalue for L1}
    \lambda_{1}^{0}=\inf\Set{\ractn{\Lcal_{1}f}{f}}{\int_{0}^{+\infty}\abs{\fct{f}{r}}^2r^5\odif{r}=1 }=0.
\end{equation}
\end{prop}
\begin{proof}
First, by differentiation \eqref{NLH scaling and translation} with respect to $z$, and taking
$\lambda =1$, $z=0$, we obtain,
\begin{equation}
\label{kl:001}
    \fct{L\frac{\partial\omega}{\partial x_j}}{x}=0,\qquad 1\leq j\leq 6.
\end{equation}

Since $\omega$ is radial, by \eqref{spherical harmonics space of degree 1},
we have,
\begin{equation}
\label{kl:002}
    \fct{\frac{\partial\omega}{\partial x_j}}{x}
    =
    6\fct{\omega'}{\abs{x}}\fct{Y_{1,j}}{\frac{x_j}{\abs{x}}},
    \qquad 1\leq j\leq 6,
\end{equation}
which, together with \eqref{spherical harmonics coefficient} implies that,
\begin{equation}
\label{kl:003}
    \fct{\omega'_{k,j}}{r}
=\begin{cases}
    0, & k=0, j=1,\\
    \fct{\omega'}{r}, & k=1, 1\leq j\leq 6,\\
    0, & k\geq 2, 1\leq j\leq \alpha_k.
\end{cases}
\end{equation}
By \Cref{theo:decomposition L}, using \eqref{kl:001}, \eqref{kl:002} and \eqref{kl:003}, we get,
\begin{equation*}
    \fct{\Lcal_{1}\omega'}{r}=0,
\end{equation*}
which implies that,
$0$ is an eigenvalue of the operator $\Lcal_{1}$ with the eigenfunction $\fct{\omega'}{r}$.
Moreover, since $\fct{\omega'}{r}<0$ for all $r\in\sts{0,+\infty}$, by \Cref{prop:Perron-Frobenius},
\begin{equation*}
    0=\inf\Set{\actn{\Lcal_{1}f}{f}}{\int_{0}^{+\infty}\abs{\fct{f}{r}}^2r^5\odif{r}=1 }
\end{equation*}
is the lowest eigenvalue  of the operator $\Lcal_{1}$, and any function
$\psi\in L^2\sts{\sts{0,+\infty},r^5}$ satisfying $\fct{\Lcal_{1}\psi}{r}=0$ must belongs to the set
$\Set{\beta\omega' }{\beta\in\R}$. This ends the proof of \Cref{prop:kernel of L1}.
\end{proof}
\begin{prop}
\label{prop:kernel of Lk}
Let $k\geq 2$ and $\Lcal_{k}$ be defined by \eqref{Lk operator}.
For any $\varrho\in L^2\sts{\sts{0,+\infty},r^5}$ satisfying $\Lcal_{1}\varrho=0$,
we have $ \varrho\equiv 0$.
\end{prop}
\begin{proof}
We argue by contradiction.
For any $\varrho\in L^2\sts{\sts{0,+\infty},r^5}$, by letting
\begin{equation*}
    \fct{f_{\rho,k}}{x}=\fct{\varphi}{\abs{x}}\fct{Y_{k,1}}{\frac{x}{\abs{x}}},
\end{equation*}
we have
\begin{equation*}
    f_{\rho,k}\in L^2\sts{\R^6},\qtq{and}\actn{f_{\rho,k}}{\omega^2}=0,
\end{equation*}
which together with \Cref{lem:nonnegative}, implies that,
\begin{equation*}
    \ractn{\Lcal_{1}\rho}{\rho}
    =
    \actn{L f_{\rho,k}}{f_{\rho,k}}\geq 0,
\end{equation*}
therefore,
\begin{equation*}
    \lambda_{k}^{0}=\inf\Set{\ractn{\Lcal_{k}\rho}{\rho}}{\int_{0}^{+\infty}\abs{\fct{\rho}{r}}^2r^5\odif{r}=1 }\geq 0.
\end{equation*}

If $\lambda_{k}^{0}=0 $ is attained or $\lambda_{k}^{0}>0 $,
then for any $\rho\in L^2\sts{\sts{0,+\infty},r^5}$ with $\rho\neq 0$,
we have $\actn{\Lcal_{k}\rho}{\rho}>0$, which contradicts with $\Lcal_{k}\rho= 0$.

If $\lambda_{k}^{0}=0 $ is attained, then by \Cref{prop:Perron-Frobenius},
the lowest eigenvalue $0$ of the operator $\Lcal_{k}$ is simple,
and the corresponding eigenfunction $\varrho$ with $\int_{0}^{+\infty}\abs{\fct{\rho}{r}}r^5\odif{r}=1$ does not
change sign on $\sts{0,+\infty}$. Without loss of generality, we may assume that
$\fct{\varrho}{r}>0$ for all $r\in\sts{0,+\infty}$.
By \eqref{lem:Lk monotone}, we obtain that,
\begin{equation*}
    \ractn{\Lcal_{1}\varrho}{\varrho}
    <
    \ractn{\Lcal_{k}\varrho}{\varrho}=0,
\end{equation*}
which contradicts with \eqref{lowest eigenvalue for L1}.
This ends the proof of \Cref{prop:kernel of Lk}.
\end{proof}
\begin{proof}[Proof of \Cref{main theorem}]
First, by \eqref{spherical harmonics coefficient}, \eqref{spherical harmonics expand},
and \Cref{theo:decomposition L}, we have,
$$f\in L^2\sts{\R^6} \qtq{satisfies} {L f}=0,$$
if and only if
\begin{equation*}
    \fct{\Lcal_{k}f_{k,j}}{r}=0,\qtq{for all} k\geq 0 \qtq{and} 1\leq j\leq\alpha_{k},
\end{equation*}
where $f_{k,j}\in L^2\sts{\sts{0,+\infty},r^5}$ is defined by \eqref{spherical harmonics coefficient},
and
\begin{align*}
        \fct{f}{x}
        =
        \sum_{k=0}^{\infty}
        \sum_{j=1}^{\alpha_k}\fct{f_{k,j}}{\abs{x}}\fct{Y_{k,j}}{\frac{x}{\abs{x}}}.
\end{align*}

Next, by \Cref{prop:kernel of L0}, \Cref{prop:kernel of L1} and \Cref{prop:kernel of Lk},
there exist real numbers $\eta$ and $\beta_j$, ($1\leq j\leq 6$) such that,
\begin{equation*}
    \fct{f}{x} = \eta \fct{\Lambda\omega}{x}
    +
    \sum_{j=1}^6\beta_j \fct{\omega'}{\abs{x}}{\frac{x_j}{\abs{x}}},
\end{equation*}
i.e.
\begin{equation*}
    f\in\spn
\left\{
        \Lambda\omega,
        ~\frac{\partial\omega}{\partial{x_1}},
        ~\frac{\partial\omega}{\partial{x_2}},
        ~\frac{\partial\omega}{\partial{x_3}},
        ~\frac{\partial\omega}{\partial{x_4}},
        ~\frac{\partial\omega}{\partial{x_5}},
        ~\frac{\partial\omega}{\partial{x_6}}
\right\}.
\end{equation*}
This ends the proof of \Cref{main theorem}.
\end{proof}

\appendix
\section{Properties of the linear operator $\Lfrak_{0}$}
\begin{prop}
\label{prop:kernel operator L0:part1}
Let $\Lfrak_{0}$ be define by \eqref{operator L0:part1}.
If $\phi\in C^2\left[0,+\infty\right)$ with $\fct{\phi}{0}\neq 0$ and $\fct{\phi'}{0}=0$
satisfies $ \fct{\Lfrak_{0}\phi}{r}=0 $, then
$\fct{\phi}{r}$ does not change sign and
\begin{equation}
\label{eq:kernel operator L0:part1}
    {\abs{\fct{\phi}{r}}}>\frac{\abs{\fct{\phi}{0}}}{4},
    \qtq{for all}
    r\in \left[0,+\infty\right).
\end{equation}
\end{prop}
\begin{proof}
First, by setting
\begin{equation}
\label{kfl0:007}
    \fct{\tilde{\phi}}{r}=\frac{2\fct{\omega}{0}}{\fct{\phi}{0}}\fct{\phi}{r},
    \qtq{and }
    \fct{\lambda}{r}=\frac{\fct{\tilde{\phi}}{r}}{\fct{\omega}{r}} .
\end{equation}
we obtain that $\tilde{\phi}$ satisfies the linear equaion $\fct{\Lfrak_{0}\tilde{\phi}}{r}=0$ as well.
Moreover, since $\phi\in C^2\left[0,+\infty\right)$ and $\omega\in C^{\infty}\left[0,+\infty\right)$,
by \eqref{kfl0:007}, we have $\lambda\in C^2\left[0,+\infty\right) $ with $\fct{\lambda}{0}>1$.
Let us define
\begin{align*}
    r^{\ast}=\sup\Set{r>0}{\fct{\lambda}{t}>1,\qtq{for all } t\in\left[0,r\right]}.
\end{align*}
By the definition of $r^{\ast}$, it is obvious that,
\begin{equation}
\label{kfl0:006}
    \fct{\lambda'}{r^{\ast}}\leq 0.
\end{equation}

We claim that
\begin{equation}
\label{kfl0:001}
    r^{\ast}=+\infty.
\end{equation}
For this,
we argue by contradiction, assuming that $r^{\ast}<+\infty$ and obtaining a contradiction with the definition of $r^{\ast}$.
By noting that $\omega$ satisfies
\begin{equation*}
    -\fct{\omega''}{r}-\frac{5}{r}\fct{\omega'}{r}
    -2\pi^{\sfrac{3}{2}}\fct{\omega^2}{r}=0,
\end{equation*}
we have,
\begin{equation}
\label{kfl0:002}
    \fct{\Lfrak_{0}\omega}{r}
    =
    -2\pi^3\fct{\omega}{r}\int_0^r\sts{ \frac{t^4}{r^4}-1 }\fct{\omega}{t}\fct{\omega}{t}t\odif{t}.
\end{equation}
By combining $\fct{\Lfrak_{0}\tilde{\phi}}{r}=0 $ with \eqref{kfl0:002}, we obtain that,
\begin{align*}
&
2\pi^3\fct{\tilde{\phi}}{r}\fct{\omega}{r}
\int_0^r\sts{ \frac{t^4}{r^4}-1 }\fct{\omega}{t}\fct{\omega}{t}t\odif{t}.
\\
=
&
-\fct{\omega }{r}\fct{\tilde{\phi}''}{r}-\frac{5}{r}\fct{\omega }{r}\fct{\tilde{\phi}'}{r}
+\fct{\tilde{\phi} }{r}\fct{\omega''}{r}+\frac{5}{r}\fct{\tilde{\phi} }{r}\fct{\omega'}{r}
\\
&
-2\pi^3\fct{\omega^2}{r}
\int_{0}^{r}\sts{\frac{t^4}{r^4}-1}\fct{\tilde{\phi}}{t}\fct{\omega}{t}t\odif{t}
+2\pi^3\fct{\tilde{\phi}}{r}\fct{\omega}{r}
\int_{0}^{r}\sts{\frac{t^4}{r^4}-1}\fct{\omega^2}{t}t\odif{t}.
\end{align*}
Moreove, it is elementary to check that,
\begin{align}
\label{kfl0:003}
    \left[
        r^{5}
        \left(
            \fct{\tilde{\phi}'}{r}\fct{\omega}{r}-\fct{\omega'}{r}\fct{\tilde{\phi}}{r}
        \right)
    \right]'
    =
    -2\pi^3r^5
    \fct{\omega^2}{r}
    \int_{0}^{r}\sts{\frac{t^4}{r^4}-1}\fct{\tilde{\phi}}{t}\fct{\omega}{t}t\odif{t},
\end{align}
which implies that,
\begin{align}
\label{kfl0:005}
    &
    \left[
        r^{5}\fct{\omega^2}{r}\fct{\lambda'}{r}
    \right]'
    =
    2\pi^3r
    \fct{\omega^2}{r}
    \int_{0}^{r}\sts{{r^4}-{t^4}}\fct{\tilde{\phi}}{t}\fct{\omega}{t}t\odif{t}.
\end{align}
By integrating \eqref{kfl0:005} from $0$ to $r^{\ast}$, we obtain that,
\begin{align*}
    \fct{\lambda'}{r^{\ast}}
    =&
    \frac{2\pi^3}{{r^{\ast}}^{5}\fct{\omega^2}{r^{\ast}}}
    \int_{0}^{r^{\ast}}
    s\fct{\omega^2}{s}
    \int_{0}^{s}\sts{{s^4}-{t^4}}\fct{\tilde{\phi}}{t}\fct{\omega}{t}t\odif{t}\odif{s}
    \\
    > &
    \frac{2\pi^3}{{r^{\ast}}^{5}\fct{\omega^2}{r^{\ast}}}
    \int_{0}^{r^{\ast}}
    s\fct{\omega^2}{s}
    \int_{0}^{s}\sts{{s^4}-{t^4}}\fct{\omega^2}{t}t\odif{t}\odif{s}.
\end{align*}
Using the fact $\fct{\omega}{r}>0$, we have,
   $ \fct{\lambda'}{r^{\ast}}
    >0$,
which contradicts with \eqref{kfl0:006}. Therefore \eqref{kfl0:001} holds.

Now, by integrating \eqref{kfl0:005} from $0$ to $r$ and an elementary calculation, we obtain that,
\begin{align*}
    \fct{\lambda}{r}-\fct{\lambda}{0}
    >
    \frac{9}{5}r^4+\frac{12}{5}r^2-\frac{12}{5}\fct{\log}{1+r^2},
\end{align*}
which, together with $\fct{\lambda}{0}>1$, implies that,
   $ \fct{\lambda}{r}>r^4+1\geq \frac{1}{2}\sts{1+r^2}^2$.
By \eqref{omega} and \eqref{kfl0:007},
we obtain that $\frac{\fct{\phi}{r}}{\fct{\phi}{0}}>\frac{1}{4}$. Therefore, \eqref{eq:kernel operator L0:part1} holds.
\end{proof}
As a consequence of \Cref{prop:kernel operator L0:part1}, we have the following corollary.
\begin{coro}
\label{coro:kernel operator L0:part1}
Let $\Lfrak_{0}$ be define by \eqref{operator L0:part1}.
If $\phi\in L^2\left(\sts{0,+\infty},r^5\right)\cap C^2\left(0,+\infty\right)$ satisfies $\fct{\Lfrak_{0}\phi}{r}=0$
with $\fct{\phi'}{0}=0$,
then $\fct{\phi}{r}=0$ for all $r\in\left[0,+\infty\right)$.
\end{coro}

\noindent \subsection*{Acknowledgements.}
G. Xu  was supported by National Key Research and Development Program of China (No. 2020YFA0712900) and by NSFC (No. 11831004). X. Tang was supported by NSFC (No. 12001284).

\end{document}